\documentclass[11pt,a4paper]{amsart}
\usepackage[utf8]{inputenc}
\usepackage[english]{babel}
\usepackage{amsmath}
\usepackage[psamsfonts]{amssymb}
\usepackage{mathrsfs}
\usepackage{amsthm}
\usepackage{enumerate}
\usepackage{color}
\usepackage{verbatim,euscript,eufrak}
\usepackage{caption}
\usepackage{graphicx}

\usepackage{tikz}
\usepackage{tikz-3dplot}

\def\Sep{\mathcal{S}^\ep}

\def\Lip{W^{1,\infty}}
\def\loc{_{\text{loc}}}

\def\toepk{\stackrel{\ep_k\to0}{\longrightarrow}}

\def\epk{{\ep_k}}

\DeclareMathOperator\curl{curl}
\DeclareMathOperator\dive{div}

\newcommand{\RR}{\mathbb R}
\newcommand{\PP}{\mathbb P}

\newcommand{\pat}{\partial_t}

\newcommand{\na}{\nabla}
\newcommand{\ue}{u_{\varepsilon}}
\newcommand{\ep}{\varepsilon}
\newcommand{\wepk}{\widetilde{u}_{\ep_{k}}}
\newcommand{\wep}{\widetilde{u}_{\ep}}
\newcommand{\ve}{v_{\ep}}
\newcommand{\vet}{\widetilde v_{\ep}}
\newcommand{\vepk}{v_{\ep_{k}}}

\newcommand{\ome}{\omega_\ep}
\newcommand{\lep}{l_\ep}

\newcommand{\norm}[1]{\left\lVert#1\right\rVert}
\newcommand{\vertiii}[1]{{\left\vert\kern-0.25ex\left\vert\kern-0.25ex\left\vert #1 
    \right\vert\kern-0.25ex\right\vert\kern-0.25ex\right\vert}}

\newcounter{comentcount}
\newcounter{teocount}
\newtheorem{lem}{Lemma}
\newtheorem{prop}{Proposition}

\newtheorem{teo}[teocount]{Theorem}  
\newtheorem{defi}{Definition}

\newtheorem{remark}{Remark}

\title[]{On the small rigid body limit in 3D incompressible flows}
\author{Jiao He}
\author{Dragoș Iftimie}
\begin{document}

\begin{abstract}
We consider the evolution of a small rigid body in an incompressible viscous fluid filling the whole space $\RR^3$.
The motion of the fluid is modelled by the Navier-Stokes equations, whereas the motion of the rigid body is described by the conservation law of linear and angular momentum. Under the assumption that the diameter of the rigid body tends to zero and that the density of the rigid body goes to infinity, we prove that the solution of the fluid-rigid body system converges to a solution of the Navier-Stokes equations in the full space without rigid body.  
\end{abstract}

\maketitle

\section{Introduction and Statement of results}

The motion of one or several rigid bodies in a liquid is a classical topic in fluid mechanics. We consider here the motion of an incompressible viscous fluid and a small smooth moving rigid body in the three-dimensional space. In this fluid-rigid body system, we will not take into account the gravity, and we suppose that the rigid body moves under the influence of the fluid. We study the asymptotic behavior of the fluid-rigid body system as the diameter of the rigid body tends to zero. 

Assume that the whole three-dimensional space is occupied by an incompressible viscous fluid of viscosity $\nu >0 $ and by a rigid body of size $\ep$. 
At the initial time,  
the domain of the rigid body $\mathcal{S}_{0}^{\ep}$ is a small non-empty smooth compact simply-connected subset of $\RR^{3}$ included in the ball $B(0,\ep)$ and $\mathcal{F}_{0}^{\ep} =\RR^{3}\setminus \mathcal{S}^{\ep}_0$ is the domain of the fluid. We also denote by $\mathcal{S}^{\ep}(t)$  the region occupied by the rigid body and by $\mathcal{F}^{\ep}(t) = \RR^{3}\setminus \mathcal{S}^{\ep}(t)$  the region occupied by the viscous fluid at time $t$.

In order to describe the motion of the rigid body, we need to specify its center of mass, that we denote by $h_\ep(t)$ and a rotation matrix $\mathcal{R}(t)  \in SO(3)$ which describes how the body rotates compared to the initial position. In other words, we have that 
$$
\mathcal{S}^{\ep}(t)= \{ x \in \RR^3 \; | \;x = h_{\ep}(t) + \mathcal{R}(t) x_0, \;x_0 \in \mathcal{S}^{\ep}_{0}\}.
$$

The velocity of  the solid particle $x(t) = h_{\ep}(t) + \mathcal{R}(t) x_0$ is given by
\begin{align*}
x'(t)&=h'_{\ep}(t) + \mathcal{R}'(t) x_0\\
&=h'_{\ep}(t) + \mathcal{R}'(t)\mathcal{R}(t)^{-1}(x-h_\ep(t))\\
&=h'_{\ep}(t) + \mathcal{R}'(t)\mathcal{R}(t)^{T}(x-h_\ep(t))
\end{align*}
where the superscript $^T$ denotes the transpose. Since  $\mathcal{R}(t)  \in SO(3)$, the matrix $\mathcal{R}'(t)\mathcal{R}(t)^{T}$ is skew-symmetric and can therefore be identified to a three-dimensional rotation vector $\ome(t)$:
\begin{equation*}
\mathcal{R}'(t)\mathcal{R}(t)^{T}z=\ome(t)\times z,\quad z\in\RR^3
\end{equation*}
where $\times $ denotes the standard cross product of vectors in $\RR^{3}$. Therefore the velocity of the solid particle $x$ is given by 
\begin{align}\label{velocitysolid}
h'_{\ep}(t) + \omega_{\ep}(t) \times (x- h_{\ep}(t)), \quad x \in \mathcal{S}^{\ep}(t).
\end{align}

 We assume that the rigid body is homogeneous of density $\rho_\ep$. We denote its total mass by $m_\ep$ so that $m^\ep = \rho_\ep |\mathcal{S}^\ep_0 |$ where $|\mathcal{S}^\ep_0 |$ is the volume of the rigid body. We also introduce  $J^{\ep}$  the matrix of inertia of the rigid body defined by 
\begin{equation*}
(J^\ep a ) \cdot b = \rho_\ep \int_{\mathcal{S}^\ep} \big( a \times( x - h_\ep(t)\big) \cdot \big( b \times( x - h_\ep(t)   \big)  dx,\;\; \text{for any} \;\; a,b \in \RR^3, 
\end{equation*}
(see \cite{gunzburger_global_2000}).

We assume that the fluid is governed by the classical Navier-Stokes equations with no-slip boundary conditions on the boundary of the rigid body, and the dynamics of the rigid body is described by the equations of the balance of linear and angular momentum. 
We  suppose that the fluid is homogeneous of constant density $1$ to simplify the notations and we denote by $u_{\ep}(t,x)$ the velocity of the fluid and by $p_{\ep}(t,x)$ the  pressure of the fluid. Moreover, we also denote by $\Sigma (u_\ep, p_\ep)$ the stress tensor of the fluid 
$$
\Sigma (u_\ep, p_\ep) = 2 \nu D(u_\ep) - p_\ep I_{3},
$$ 
where $I_{3}$ is identity matrix of order $3$ and $D(u_\ep)$ is the deformation tensor
\begin{equation}\label{symmetric}
D(u_\ep):=\frac{1}{2}\bigl(\frac{\partial u_{\ep,i}}{\partial x_{j}}+ \frac{\partial u_{\ep,j}}{\partial x_{i}} \bigr)_{i,j} \quad i,j= 1,2,3. 
\end{equation}

With the notation introduced above, we have the following mathematical formulation for the fluid-rigid body system (see \cite{desjardins_existence_1999}, \cite{gunzburger_global_2000} and \cite{serre_chute_1987}):
\begin{itemize}
\item Fluid equations:
\begin{flalign}{\label{ns-ob1}}
\begin{split}
\frac{\partial u_{\ep}}{\partial t} +(u_{\ep} \cdot \nabla )u_{\ep}- \nu \Delta u_{\ep} + \nabla p_{\ep} =0 \;\; \text{for}\; t \in (0, +\infty),\; x \in \mathcal{F}^{\ep}(t). \\
\dive u_{\ep} = 0   \;\; \text{for}\; t \in (0, +\infty),\; x \in \mathcal{F}^{\ep}(t).
\end{split}
\end{flalign}
\item Rigid body equations:
\begin{align}
m^{\ep} h''_{\ep}(t) = - \int_{\partial \mathcal{S}^{\ep}(t)} \Sigma(u_\ep,p_\ep) n_\ep ds \;\;\; \text{for}\; t \in (0, +\infty).{\label{solid1}} \\
(J^{\ep}\omega_{\ep})'(t) = - \int_{\partial \mathcal{S}^{\ep}(t)} (x-h_{\ep}) \times (\Sigma (u_\ep,p_\ep) n_\ep) ds \;\;\; \text{for}\; t \in (0, +\infty). {\label{solid2}}
\end{align}
\item Boundary conditions:
\begin{flalign}\label{boundarycon1}
\begin{split}
u_{\ep}(t,x)= h_{\ep}'(t)+ \omega_{\ep}(t) \times \big(x-h_{\ep}(t)\big), \; \text{for}\; t \in (0, +\infty), x \in \partial \mathcal{S}^{\ep}(t).\\
\lim_{|x|\rightarrow \infty} u_{\ep}(t,x) =0 \;\;\; \text{for}\; t \in [0, +\infty).
\end{split}
\end{flalign}
\end{itemize}

In the above system, we have denoted by $n_{\ep}(t,x)$ the unit normal vector to $\partial\mathcal{S}^{\ep}$ pointing outside the fluid domain $\mathcal{F}^{\ep}$. The first line in \eqref{boundarycon1} is the Dirichlet boundary condition: the fluid velocity and the solid velocity must agree on the boundary of the body.

The system \eqref{ns-ob1}-\eqref{boundarycon1} should be completed by some initial conditions. As mentioned at the beginning, we assume that the initial position of the center of mass of the rigid body is in the origin. We denote by $\ue^0$ the initial fluid velocity:
\begin{align}{\label{initialcon1}}
 u_{\ep}(0, x) = \ue^0, \quad
h_{\ep} (0)=0, \quad 
h'_{\ep} (0) = \lep^0, \quad
\ome(0) = \ome^0.
\end{align}

The coupled system satisfies some $L^2$ energy estimates at least at the formal level. Taking the inner product of \eqref{ns-ob1} with $\ue$, integrating the result by parts and using the equations \eqref{solid1} and \eqref{solid2}, we get the following energy estimate (see \cite{gunzburger_global_2000}):
\begin{multline}\label{energinequality}
\|u_\ep(t)\|_{L^2(\mathcal{F}^{\ep}(t))}^2 + m^\ep |h'_\ep(t)|^2 + \left(J^\ep \ome(t) \right) \cdot \ome(t) + 4\nu\int_0^t \|D(u_\ep)\|_{L^2(\mathcal{F}^{\ep}(t))}^2 \\
\leq \|\ue^0\|_{L^2(\mathcal{F}^{\ep}_0)}^2 + m_\ep |\lep^0|^2 + \left(J^\ep \ome^0 \right) \cdot \ome^0.
\end{multline}
The sum of the first three terms on the left-hand side of \eqref{energinequality} is called the kinetic energy of the system at time $t$, while the forth term is called viscous dissipation. Obviously the initial kinetic energy is the right-hand side of \eqref{energinequality}.

Over the last few years, there were a lot of works dealing with the well-posedness of the fluid-rigid body system by using energy estimates. Both weak finite energy solutions (Leray solutions) and strong $H^1$ solutions were constructed. When the fluid is enclosed in a bounded region, the existence of solutions is proved under  some constraints on the collisions between the rigid body and the boundary of the domain. When the  domain of motion is the whole of $\RR^3$ there is of course no such constraint. We give some references below but we would like to say that this list is not exhaustive.
The existence of weak Leray solutions have been proved in \cite{conca_existence_2000},  \cite{desjardins_weak_2000}, \cite{gunzburger_global_2000} and \cite{serre_chute_1987} (see also the references therein). 
We refer to \cite{cumsille_wellposedness_2008}, \cite{desjardins_existence_1999} and \cite{galdi_strong_2002} for results about strong $H^1$ solutions. The vanishing viscosity limit was considered in \cite{sueur_kato_2012}. Let us also mention that the case of the dimension two was also considered in the literature, see for example \cite{desjardins_existence_1999}, \cite{hoffmann_motion_1999}, \cite{san_martin_global_2002} and \cite{takahashi_global_2004}.

The initial conditions should satisfy the following compatibility conditions (see \cite{desjardins_existence_1999}):
\begin{equation} {\label{inicon}}  
\begin{aligned}  
& \ue^0 \in L^{2}(\mathcal{F}^\ep_0), \quad  \dive \ue^0=0\;\; \text{in} \;\; \mathcal{F}^\ep_0 ,\\
& \ue^0 \cdot n_\ep = (\lep^0 + \ome^0 \times x ) \cdot n_\ep \;\; \text{on} \;\; \partial \mathcal{S}_{0}^{\ep}.
\end{aligned}
\end{equation}
The second condition above is a weak version of the Dirichlet boundary condition in which only the normal components of the fluid velocity and of the solid velocity must agree on the boundary of the obstacle. This is in agreement with the usual theory of Leray solutions of the Navier-Stokes equations where the initial velocity is assumed to be only tangent to the boundary. 

Before stating a result of existence of weak solutions for the motion of a rigid body in a fluid, let us introduce the global density and the global velocity, defined on the whole of $\RR^3$. The fluid is homogeneous of constant density $1$ while the rigid body is of density $\rho_\ep$, so we can define the global density  $\widetilde{\rho}_{\ep}(t,x)$ as follows:
\begin{equation*}
\widetilde{\rho}_{\ep}(t,x) = \chi_{\mathcal{F}^\ep(t)}(x) + \rho_\ep \chi_{\mathcal{S}^\ep(t)}(x), \;\; \text{for}\;\;  x \in \RR^3
\end{equation*}
where we denote $\chi_A$ denotes the characteristic function of the set $A$. 
Moreover, recalling the formula for the  velocity of the rigid body, see \eqref{velocitysolid}, one may define a global velocity $\wep$ by 
\begin{equation*}
\wep(t,x)=
\begin{cases}
\ue(t,x)&\quad\text{if }x\in \mathcal{F}^\ep(t)\\
h'_{\ep}(t)+\omega_{\ep}(t)\times\left(x-h_{\ep}(t) \right) &\quad\text{if }x\in\mathcal{S}^\ep(t). 
\end{cases}
\end{equation*}
Clearly, by conditions \eqref{inicon}, we know that 
\begin{equation*}
\wep^0 \in L^{2}(\RR^{3}), \quad \dive \wep^0 = 0\;\; \text{in} \;\; \RR^{3}.\\
\end{equation*}

Motivated by the energy estimates \eqref{energinequality} and by the construction of $\widetilde{\rho}_{\ep}$ and $\wep$ we introduce the following notion of weak solution (see \cite{conca_existence_2000},  \cite{desjardins_weak_2000}, \cite{gunzburger_global_2000} and \cite{serre_chute_1987}).
\begin{defi}{\label{definitionweaksol}}
A triplet $(\wep, h_{\ep}, \omega_{\ep})$ is a weak Leray solution of \eqref{ns-ob1}--\eqref{initialcon1}, if
\begin{itemize}
\item $\wep, h_\ep, \ome$ satisfying
\begin{equation*}
h_\ep \in W^{1,\infty}(\RR_+; \RR^3),\quad  \ome \in L^{\infty} (\RR_+; \RR^3),
\end{equation*}
\begin{equation*}
\ue\in L^{\infty}\left(\RR_+; L^{2}(\mathcal{F}^\ep)\right) \cap L^{2}\loc \left(\RR_+; H^{1}(\mathcal{F}^\ep)\right),\quad \wep\in C^0_w\left(\RR_+; L^{2}(\RR^3)\right);
\end{equation*}
\item $\wep$ is divergence free in the whole of $\RR^3$ with $D\wep(t,x)=0$ in $\mathcal{S}^\ep(t)$;
\item $\wep$ verifies the equation in the following sense:
\begin{multline*}
-\int_{0}^{\infty} \int_{\RR^{3}} \widetilde{\rho}_{\ep} \wep \cdot \big(\pat \varphi_{\ep} + (\wep \cdot \na) \varphi_{\ep} \big) + 2 \nu \int_{0}^{\infty} \int_{\RR^{3}} D (\wep) : D (\varphi_{\ep}) \\
 = \int_{\RR^{3}} \widetilde{\rho}_{\ep} (0) \wep^0 \cdot \varphi_{\ep}(0). 
\end{multline*}
for any test function $\varphi_{\ep} \in W^{1,\infty}\left(\RR_+; H^{1}_{\sigma}(\RR^{3})\right)$, compactly supported in $\RR_+\times\RR^3$  such that $D\varphi_\ep(t,x)=0$ in $\mathcal{S}^\ep(t)$.
\end{itemize}
\end{defi}

One has the following result of existence of weak solutions of the initial-boundary value problem  \eqref{ns-ob1}--\eqref{initialcon1} in the sense defined above (see \cite{conca_existence_2000},  \cite{desjardins_weak_2000}, \cite{gunzburger_global_2000} and \cite{serre_chute_1987}). 

\begin{teo}\label{Existence}
Let $\wep^0\in L^2(\RR^3)$ be divergence free and such that  $D\wep^0=0$ in $\mathcal{S}^\ep_0$. 
Then, there exists at least one  global weak solution $(\wep, h_{\ep}, \omega_{\ep})$ of the initial-boundary value problem \eqref{ns-ob1}--\eqref{initialcon1} in the sense of Definition \ref{definitionweaksol}. Moreover, $\wep$ satisfies the following energy estimate:
\begin{align}{\label{energyineq}}
\int_{\RR^{3}} \widetilde{\rho}_{\ep} |\wep(t)|^{2}  + 4 \nu \int_{0}^{t} \int_{\RR^{3}} |D (\wep)|^{2} \leq \int_{\RR^{3}} \widetilde{\rho}_{\ep}(0) |\wep^0|^{2} \quad \forall t>0.
\end{align}
\end{teo}

We now let $\ep\to0$ and wish to find the limit of the solution given in Theorem \ref{Existence}. Let us first review the literature available on related results.

In dimension two the literature is richer. Iftimie, Lopes Filho and Nussenzveig Lopes \cite{iftimie_two-dimensional_2006} proved convergence towards the Navier-Stokes equations in $\RR^2$ in the case when the rigid body does not move. Lacave \cite{lacave_two_2009} considered the case of a thin obstacle tending to a curve. Recently, Lacave and Takahashi \cite{lacave_small_2017} considered a small disk moving under the influence of a two-dimensional viscous incompressible fluid. Under the condition that the density of the solid is independent of $\ep$ and assuming that the initial data is sufficiently small, they used the $L^{p}-L^{q}$ decay estimates of the semigroup associated to the fluid-rigid body system to deduce the convergence towards the solution of the Navier-Stokes equations in $\RR^{2}$. In \cite{he_small_2018}, the authors extended the result of  \cite{lacave_small_2017} to the case of arbitrary shape of the body and with no restriction on the size of the initial data but assuming that the  density of the obstacle is large. 

In  dimension three, Iftimie and Kelliher \cite{iftimie_remarks_2009} considered the case of a fixed obstacle and proved convergence towards the Navier-Stokes equations in $\RR^3$.  Lacave \cite{lacave_3d_2015} considered more general shrinking obstacles (for instance shrinking to a curve) but still fixed. When the obstacle is moving with the fluid, the limit $\ep\to0$ was considered in \cite{silvestre_motion_2014} in the case when the rigid body is a ball. Unfortunately, the elliptic estimates in that paper, see \cite[Theorem 3.1]{silvestre_motion_2014}, are not correct as was observed in \cite[Subsection 2.1]{chipot_limits_2014-1}.

As far as we know, Theorem \ref{th1} below is the first result on the limit $\ep\to0$ in dimension three for a moving obstacle. We will essentially show that if  the density of the rigid body goes to infinity, then the energy estimates are sufficient to pass to the limit in the weak formulation by using a truncation procedure. We obtain then the convergence of the solutions constructed in Theorem \ref{Existence} to a solution of the Navier-Stokes equations in $\RR^3$ under the assumption that the initial data $\wep^0$ is bounded in $L^2$. We do not need to impose any small data condition or any restriction on the shape of the body.

Let us now state  the main result of this paper.
\begin{teo}{\label{th1}}
Let $\wep^0\in L^2(\RR^3)$ be divergence free and such that  $D\wep^0=0$ in $\mathcal{S}^\ep_0$. We assume that
\begin{itemize}
\item   $ \Sep_0 \subset B(0,\ep)$;
\item the mass $m^\ep$ of the rigid body satisfies that 
\begin{equation}\label{hymass}
\frac{m^\ep}{\ep^3} \to \infty \quad \text{as } \ep \to 0; 
\end{equation}
\item $\wep^0$ converges weakly in $L^2(\RR^3)$ towards some $u_0$.
\item $\sqrt{m^\ep}h'_\ep(0)$ and $\left(J^\ep \ome^0 \right) \cdot \ome^0$ are bounded uniformly in $\ep$.
\end{itemize}
Let $(\wep , h_{\ep}, \omega_{\ep})$  be the global solution  of the system \eqref{ns-ob1}--\eqref{initialcon1} given by Theorem \ref{Existence}.
Then there exists a subsequence $\wepk$ of $\wep$ which converges
 $$\wepk \rightharpoonup u \quad \text{weak$\ast$ in }  L^\infty\left(\RR_+;L^2(\RR^{3})\right)\text{ and weakly in } L^2\loc \left(\RR_+;H^1(\RR^{3})\right)$$
towards a solution $u$ of the Navier-Stokes equations in $\RR^3$ in the sense of distributions with initial data $u_0(x)$.

Moreover, suppose in addition that $\wep^0$ converges strongly in $L^2(\RR^3)$ to $u_0(x)$  and that both  $\sqrt{m^\ep}h'_\ep(0)$ and $\left(J^\ep \ome^0 \right) \cdot \ome^0$ converge to 0 as $\ep\to0$. Then the limit solution $u$ satisfies the  energy estimate 
\begin{equation*}
\forall t\geq0\quad \norm{u(t)}_{L^{2}(\RR^{3})}^{2}+4\nu \int_{0}^{t} \norm{D(u)}^{2}_{L^{2}(\RR^{3})} \leq \norm{u_0}^{2}_{L^{2}(\RR^{3})}. 
\end{equation*}
\end{teo}

Let us give a few remarks on the hypotheses of the Theorem above. First, if  the rigid body shrinks isotropically to a point, then hypothesis \eqref{hymass} means that the density $\rho_\ep$ of the rigid body tends to infinity as $\ep \to 0$. If the rigid body does not shrink isotropically to a point, then condition \eqref{hymass} is stronger than simply saying that the density of the rigid body goes to infinity. Indeed, since  $ \Sep_0 \subset B(0,\ep)$ we have that $| \Sep_0 |\leq \frac{4\pi}3\ep^3$ so $\rho_\ep=\frac{m^\ep}{| \Sep_0 |}\geq \frac{3m^\ep}{4\pi\ep^3}\to\infty$ as $\ep\to0$.

Next, the weak convergence of  $\wep^0$ in $L^2(\RR^3)$ implies its boundedness in $L^2(\RR^3)$. Together with the hypothesis that  $\sqrt{m^\ep}h'_\ep(0)$ and $\left(J^\ep \ome^0 \right) \cdot \ome^0$ are bounded uniformly in $\ep$ this implies that the right-hand side of \eqref{energinequality} is bounded. Then \eqref{energyineq} implies that $\sqrt{\widetilde{\rho}_{\ep}}\wep$ is bounded in $L^\infty(\RR_+;L^2(\RR^3))$ and $D(\wep)$ is bounded in $L^2(\RR_+;L^2(\RR^3))$. We observed above that $\rho_\ep\to\infty$ so we can assume that $\rho_\ep\geq1$. Then we have that $\widetilde\rho_\ep\geq1$ (recall that $\widetilde\rho_\ep=1$ in the fluid region and $\widetilde\rho_\ep=\rho_\ep$ in the solid region) so $\wep$ is bounded in $L^\infty(\RR_+;L^2(\RR^3))$. We infer that 
\begin{equation}\label{boundl2}
\wep (t,x) \;\text{is bounded in}\; L^{\infty}\left(\RR_+; L^{2}(\RR^3)\right) \cap L^{2}\loc \left(\RR_+; H^{1}(\RR^3)\right)
\end{equation}
and we will see that this is all we need to pass to the limit in our PDE. We  require neither the Dirichlet boundary conditions nor the special form of $\wep$ inside the rigid body. All we need is the above boundedness and the fact that the Navier-Stokes equations are satisfied in the exterior of the ball $B(0,\ep)$. More precisely, we can prove the following more general statement. 
\begin{teo}{\label{th2}}
Let $\ve (t,x ) $  be a divergence free vector field bounded independently of $\ep$ in
\begin{equation*}
 L\loc ^{\infty}(\RR_+; L^{2}(\RR^{3})) \cap L^{2}\loc(\RR_+; H^{1}(\RR^{3}))\cap C^0_w(\RR_+; L^{2}(\RR^{3})).
\end{equation*}
We make the following assumptions:
\begin{itemize}
\item The vector field  $\ve$ verifies the Navier-Stokes equations
  \begin{equation}\label{nsequation}
\partial_t\ve-\nu\Delta\ve+\ve\cdot\nabla\ve=-\nabla\pi_\ep 
  \end{equation}
in the exterior of the ball $B(h_\ep(t),\ep)$ with initial data  $\ve(0,x)$ in the following sense:
\begin{equation*}
-\int_{0}^{\infty} \int_{\RR^{3}} \ve \cdot \pat \varphi + \nu \int_{0}^{\infty} \int_{\RR^{3}} \na \ve : \na \varphi + \int_{0}^{T} \int_{\RR^{3}} \ve \cdot \na \ve \cdot \varphi  = \int_{\RR^{3}} \ve(0) \cdot \varphi(0)
\end{equation*}
for every test function $\varphi\in W^{1,\infty}\bigl(\RR_+\times \RR^3\bigl)$ which is divergence free, compactly supported in $\RR_+\times \RR^3$ and such that for all $t$ the function $x\mapsto \varphi(t,x)$ is smooth and compactly supported in the set $\{|x-h_\ep(t)|>\ep\}$.
\item The initial data $\ve(0,x)$ is divergence free, square integrable and converges weakly to some $v_0(x)$  in $L^2(\RR^3)$.
\item The center of the ball verifies $h_\ep\in \Lip(\RR_+;\RR^3)$ and $\ep^\frac{3}{2} h'_\ep(t)\to 0 $ strongly in $L^\infty\loc (\RR_+)$ when $\ep \to 0$.
\end{itemize}
Then there exists a subsequence of $\ve$ which converges weak$\ast$ in $L^{\infty}\loc(\RR_+; L^{2}(\RR^{3}))$ and weakly in  $L^{2}\loc(\RR_+; H^{1}(\RR^{3}))$ to a solution $v$ of the Navier-Stokes equations in $\RR^3$ in the sense of distributions with initial data $v_0(x)$. 

Moreover, if we assume in addition that $\ve(0,x)$ converges strongly in $L^2$ to $v_0(x)$ and that the following energy estimate holds true for $\ve$
\begin{equation}\label{ineqvep}
\forall t\geq0\quad \norm{\ve(t)}_{L^{2}(\RR^3\setminus B(h_\ep(t),\ep))}^{2}+4\nu \int_{0}^{t} \norm{D(\ve)}^{2}_{L^{2}(\RR^3\setminus B(h_\ep(t),\ep))} \leq \norm{\ve(0)}^{2}_{L^{2}(\RR^3)} + o(1)
\end{equation}
as $\ep\to0$, then the limit solution $v$ satisfies the following energy estimate 
\begin{align}\label{ineq_vep}
\forall t\geq0\quad \norm{v(t)}_{L^{2}(\RR^{3})}^{2}+4\nu \int_{0}^{t} \norm{D(v)}^{2}_{L^{2}(\RR^{3})} \leq \norm{v_0}^{2}_{L^{2}(\RR^{3})}. 
\end{align}

\end{teo}

Let us observe that Theorem \ref{th1} follows from Theorem \ref{th2} applied for $\ve (t,x)= \wep (t,x) $. Indeed, from \eqref{boundl2} we have that $\ve$ is bounded in $L\loc ^{\infty}(\RR_+; L^{2}(\RR^{3})) \cap L^{2}\loc(\RR_+; H^{1}(\RR^{3}))$. 
Next we obviously have that $\wep$ verifies the Navier-Stokes equations in the exterior of the ball $B(h_\ep(t),\ep)$ and so does $\ve$. We observed that the right-hand side of \eqref{energinequality} is bounded so $\sqrt{m^\ep}h'_\ep$ is bounded. From \eqref{hymass} we infer that  $\ep^\frac{3}{2} h'_\ep(t)\to 0 $ strongly in $L^\infty(\RR_+)$ when $\ep \to 0$. Then all the hypothesis of the first part of Theorem \ref{th2} is verified and the first part of Theorem \ref{th1} follows. 

Let us now assume in addition the $L^2$ strong convergence of $\wep^0$ towards $u_0$ and let us prove \eqref{ineqvep}. We have the $L^2$ strong convergence of $\ve(0,x)$ to $v_0$. Recalling that the matrix of inertia is non-negative we can ignore the second and the third terms in \eqref{energinequality} to estimate 
\begin{align*}
\|\ve(t)\|_{L^2(\RR^3\setminus B(h_\ep(t),\ep))}^2 +& 4\nu\int_0^t \|D(\ve)\|_{L^2(\RR^3\setminus B(h_\ep(t),\ep))}^2\\
&=\|u_\ep(t)\|_{L^2(\RR^3\setminus B(h_\ep(t),\ep))}^2 + 4\nu\int_0^t \|D(u_\ep)\|_{L^2(\RR^3\setminus B(h_\ep(t),\ep))}^2\\
&\leq \|u_\ep(t)\|_{L^2(\mathcal{F}^{\ep}(t))}^2 + 4\nu\int_0^t \|D(u_\ep)\|_{L^2(\mathcal{F}^{\ep}(t))}^2 \\
&\leq \|\ue^0\|_{L^2(\mathcal{F}^{\ep}(t))}^2 + m_\ep |\lep^0|^2 + \left(J^\ep \ome^0 \right) \cdot \ome^0\\
&\leq\|\ve(0)\|_{L^2(\RR^3)}^2 + m_\ep |\lep^0|^2 + \left(J^\ep \ome^0 \right) \cdot \ome^0.
\end{align*}
By hypothesis $m_\ep |\lep^0|^2 + \left(J^\ep \ome^0 \right) \cdot \ome^0\to0$ so  \eqref{ineqvep} follows.

The passing to the limit stated in Theorem \ref{th2} uses the boundedness of  $\ve$ in the energy space $ L^{\infty}\left(\RR_+; L^{2}(\RR^{3})\right) \cap L^{2}\loc\left(\RR_+; H^{1}(\RR^{3})\right)$ and the construction of a cut-off $\varphi_\ep$  supported in the exterior of the ball $\overline{B}(h_\ep(t),\ep)$. We multiply \eqref{nsequation} with the cut-off $\varphi_{\ep}$, and then pass to the limit by means of classical compactness methods. The main obstruction is that, when the rigid body moves under the influence of the fluid, not only the velocity depends on time, but also the cut-off function. Time derivative estimates for $\ve$ are not easy to obtain and, once obtained, it is not easy to pass to the limit in the term with the time derivative.

The paper is organized as follows. In Section \ref{sectnot}, we introduce some notation and present some preliminary results. The construction of the cut-off near the rigid body is given in Section \ref{defcutof}. We show the strong convergence by means of temporal estimates in Section \ref{temp} and pass to the limit to conclude our proof in Section \ref{paslim}.

\section{Notation and Preliminary results}
\label{sectnot}
In this section, we will introduce some notations and preliminary results.

For a sufficiently regular vector field $u: \RR^3 \to \RR^3$, we denote by $\na u$ the second order tensor field whose components $(\na u)_{ij}$ are given by $\partial u_{j}/\partial x_{i}$, and by $D(u)$ the symmetric part of $\na u$ (see \eqref{symmetric}). The double dot product $M:N$ of two matrices $M=(m_{ij})$ and $N=(n_{ij})$ denotes the quantity $\sum_{i,j}m_{ij}n_{ij}$. 

For function spaces, we shall use standard notations $L^{p}$ and $H^{m}$ to denote the usual Lebesgue and Sobolev spaces. $C^{m}_{b}$ denotes the set of bounded functions whose first $m$ derivatives are bounded functions.  We add subscripts $0$ and $\sigma$ to these  spaces to specify that their elements are compactly supported and divergence free, respectively. For instance, the notation $C^\infty_{0,\sigma}$ defines the space of smooth, compactly supported and divergence free vector fields on $\RR^{3}$. 

In addition, unless we specify the domain,  all function spaces and norms are considered to be taken on $\RR^{3}$ in the $x$ variable. For the $t$ variable, we use the notation $\RR_+ = [0, \infty)$ and emphasize that the endpoint $0$ belongs to $\RR_+$. 
Throughout this article, we denote by  $C$ a generic constant whose value can change from one line to another.

Let  $\varphi\in C^1_b (\RR_+;C^\infty_{0,\sigma})$. The stream function $\psi$ of $\varphi$ is defined by the following formula:
\begin{equation*}
\psi(x)=-\int_{\RR^{3}} \frac{x-y}{4\pi |x-y| ^{3}} \times \varphi(y) dy
\end{equation*}
where $\times$ denotes the standard cross product of vectors in $\RR^{3}$.

Because $\varphi$ is divergence free, we have that   $\curl \psi = \varphi$ and $\psi = \curl \Delta^{-1} \varphi$. Furthermore, $\psi$ is smooth, $\psi\in C^1_b (\RR_+;C^\infty)$, and vanishes at infinity. Moreover, we have the following well-known estimate: 
\begin{align}\label{var+psi}
\norm{\na \psi(t, \cdot)}_{H^2} \leq C \norm{\varphi (t, \cdot) }_{H^2}.
\end{align} 

In our case, in order to deal with the singularity in $h_\ep$, we need to have a stream function vanishing in $h_\ep$. The stream function $\psi$ defined above has no reason to vanish in $h_\ep$, so we are led to introduce a modified stream function $\psi_\ep$. We define  
\begin{equation}\label{def_psiep}
\psi_\ep(t,x)=\psi(t,x)-\psi(t, h_{\ep}(t)).
\end{equation}

Clearly $\psi_\ep(t,h_\ep(t))=0$. We collect in the following lemma some useful properties of the modified stream function.
\begin{lem}\label{lempsi}
	Let  $\varphi\in C^1_b (\RR_+;C^\infty_{0,\sigma})$ and define the  modified stream function $\psi_\ep$ as in \eqref{def_psiep}. We have that:
	\begin{enumerate}[(i)]
		\item $\psi_\ep\in \Lip (\RR_+;C^\infty)$ and $ \curl \psi_\ep = \varphi$.
		\item There exists a universal constant $C>0$ such that for all $R>0$  we have that 
		\begin{gather}
		\norm{\psi_\ep(t,\cdot)}_{L^\infty(B(h_\ep(t),R))}\leq C  R \norm{\varphi(t,\cdot)}_{H^{2}} \label{psiep2}\\
		\intertext{for all $t\geq0$ and}
		\norm{\partial_t\psi_\ep(t,\cdot)}_{L^\infty(B(h_\ep(t),R))}\leq C \left(   R \norm{\partial_t\varphi(t,\cdot)}_{H^{2}}+ |h_\ep'(t)|\norm{\varphi(t,\cdot)}_{H^{2}} \right)\notag
		\end{gather}
for almost all $t\geq0$.
	\end{enumerate}
\end{lem}
\begin{proof}
	We observe first that $\curl \psi_\ep = \curl \psi= \varphi$. Moreover,  $\psi\in C^1_b (\RR_+;C^\infty)$ and $h_\ep\in \Lip (\RR_+)$ imply that $\psi_\ep\in \Lip (\RR_+;C^\infty)$. This proves part {\it (i)}.
	
	Next,  
	to prove {\it(ii)} we  use the mean value theorem to estimate
	\begin{align*}
	\norm{\psi_\ep(t,x)}_{L^\infty(B(h_\ep(t),R))}
	& = \norm{\psi(t,x)- \psi(t,h_\ep(t))}_{L^\infty(B(h_\ep(t),R))}\\
	& \leq |x-h_\ep(t)|\norm{ \na \psi(t,x)}_{L^\infty(B(h_\ep(t),R))}\\
	& \leq R \norm{\na \psi }_{H^2}\\
	&\leq C  R \norm{\varphi}_{H^2}
	\end{align*}
	where we used relation \eqref{var+psi}. This proves \eqref{psiep2}.
	
	We recall now that $h_\ep$ is Lipschitz in time so it is almost everywhere differentiable in time. Let $t$ be a time where $h_\ep$ is differentiable. We write
	\begin{align*}
	\pat\psi_{\ep }(t,x)
	&=\pat (\psi(t,x)-\psi(t,h_\ep(t)))\\
	&=\pat\psi(t,x)-\pat\psi(t,h_\ep(t))-h_\ep'(t)\cdot\nabla\psi(  t,h_\ep(t)).
	\end{align*}
We can bound
	\begin{align*}
	\|\partial_t\psi_\ep(t,\cdot)\|_{L^\infty(B(h_\ep(t),R))}
	&\leq \|\pat\psi(t,x)-\pat\psi(t,h_\ep(t))\|_{L^\infty(B(h_\ep(t),R))}+  |h_\ep'(t)|\norm{\nabla\psi(t,\cdot)}_{L^\infty}\\
	& \leq |x-h_\ep(t)| \norm{\partial_t \na \psi(t,\cdot)}_{L^\infty} + |h_\ep'(t)|\norm{\nabla\psi(t,\cdot)}_{L^\infty} \\
	& \leq C \left( R\norm{\partial_t\nabla\psi(t,\cdot)}_{H^2}+ |h_\ep'(t)|\norm{\nabla\psi(t,\cdot)}_{H^2} \right)\\
	& \leq C\left(  R\norm{\partial_t\varphi(t,\cdot)}_{H^2}+  |h_\ep'(t)|\norm{\varphi(t,\cdot)}_{H^2}\right).
	\end{align*}
	This completes the proof of the lemma.
\end{proof}

\section{Cut-off near the rigid body}
\label{defcutof}

In this section, we will construct a cut-off $\varphi_{\ep}$ near the rigid body, which will be used as a test function in the procedure of passing to the limit in Section \ref{paslim}.

Firstly, we construct a cut-off function $\eta_\ep(t,x)  $ near the ball $B(h_\ep(t), \ep)$. Let $\eta(x) \in C^\infty(\RR^{3};[0,1])$ be a function such that
\begin{equation*}
 \eta(x) :\RR^{3}\to [0,1], \quad \eta(x)  =
\begin{cases}
0&\text{ if }|x|\leq \frac32\\
1&\text{ if }|x|\geq 2
\end{cases}
\end{equation*}

The function  $\eta(x)$ is a cut-off function in the neighborhood of the unit ball $B(0,1)$. A cut-off $\eta_\ep(t,x)$ in the neighborhood of the domain $B(h_\ep(t),\ep)$ is the following function
\begin{equation}\label{defetaep}
\eta_\ep(t,x)=\eta \left(\frac{ x-h_\ep(t)}{\ep}\right)  =
\begin{cases}
0&\text{ if }|x-h_\ep(t)|\leq \frac32\ep\\
1&\text{ if }|x-h_\ep(t)|\geq 2\ep.
\end{cases}.  
\end{equation}
Notice that $\eta_\ep(t,x)$ is a space-time function while the function $\eta(x) $ only has a space variable. We state some properties of this new cut-off in the following lemma.

\begin{lem}\label{etaep}
The cut-off function $\eta_\ep$ satisfies
\begin{enumerate}[(i)]
\item $\eta_\ep\in \Lip (\RR_+;C^\infty)$;
\item $\eta_\ep$ vanishes in the neighborhood of the ball $B(h_\ep(t),\ep)$;
\item For any real number $q\geq1$ there exists a constant $C=C(q)$ such that 
\begin{gather*}
\norm{ \eta_\ep (t,\cdot)}_{L^\infty}=1, \quad \norm{\eta_\ep (t,\cdot)-1}_{L^q} \leq  C \ep^{\frac{3}{q}},\quad
 \norm{\nabla \eta_\ep (t,\cdot)}_{L^q} \leq C \ep^{\frac{3-q}{q}}, \quad
  \norm{\nabla^2\eta_\ep (t,\cdot)}_{L^q} \leq C \ep^{\frac{3-2q}{q}} .
\end{gather*}
\end{enumerate} 
\end{lem}

\begin{proof}
Since $h_\ep$ is Lipschitz part {\it(i)} follows immediately. Part  {\it (ii)} is also obvious. We prove now part {\it(iii)}. 

Clearly $\norm{ \eta_\ep (t,\cdot)}_{L^\infty}=1$. Next
\begin{gather*}
\norm{\eta_{\ep}(t,x)-1}_{L^{q}}
=  \norm{\eta\left(\frac{x-h_{\ep} (t)}{\ep}\right)-1}_{L^{q}}
=  \ep^{\frac{3}{q}} \norm{\eta(x) -1}_{L^{q}} \leq C \ep^{\frac{3}{q}}.
\end{gather*}
Notice that $\na \eta (x)$ and $\na^2 \eta(x)$ are bounded functions supported in the annulus $\{\frac32 < |x|< 2 \}$. So
\begin{gather*}
\norm{\na \eta_{\ep}(t,x)}_{L^{q}} 
=   \norm{ \frac{1}{\ep} \na \eta \left(\frac{x-h_{\ep}(t)}{\ep}\right)}_{L^{q}} 
=  \ep^{\frac{3-q}{q}} \norm{\na \eta(x) }_{L^{q}} \leq C \ep^{\frac{3-q}{q}},\\
\norm{\na^{2}\eta_{\ep} (t,x) }_{L^{q}}
=  \norm{  \frac{1}{\ep^2} \na^2 \eta \left(\frac{x-h_{\ep}(t)}{\ep}\right)}_{L^{q}} 
= \ep^{\frac{3-2q}{q}} \norm{\na^{2} \eta(x) }_{L^{q}} \leq C \ep^{\frac{3-2q}{q}}.
\end{gather*}
This completes the proof of the lemma.
\end{proof}

Given a test function  $\varphi\in C^1_b (\RR_+;C^\infty_{0,\sigma})$, we use the cutoff $\eta_\ep$ and the modified stream function $\psi_\ep$  defined in Section \ref{sectnot} (see relation \eqref{def_psiep}) to construct a new test function  $\varphi_\ep$  which vanishes in the neighborhood of the ball   $B(h_\ep(t),\ep)$. We define 
\begin{equation}\label{cutof}
\varphi_{\ep}=\curl(\eta_{\ep} \psi_\ep).
\end{equation}

We notice that this new test function $\varphi_\ep$ depends on time even if $\varphi$ is assumed to be constant in time. We state some properties of $\varphi_\ep$ in the following lemma:
\begin{lem}{\label{lemmaphi}}
The test function $\varphi_\ep$ has the following properties:
\begin{enumerate}[(i)]
\item $\varphi_{\ep}\in \Lip (\RR_+;C^\infty_{0,\sigma})$ and  $\varphi_\ep$   vanishes in the neighborhood of   $B(h_\ep(t),\ep)$;
\item for all $T>0$ we have that $\varphi_{\ep} \to \varphi$ strongly in $L^\infty(0,T;H^1)$ as $\ep\to0$;
\item there exists a universal constant $C$ such that for all $T>0$
  \begin{equation*}
\|\varphi_{\ep}\|_{L^\infty(0,T;H^{1})} \leq C \|\varphi\|_{L^\infty(0,T;H^2)}.    
  \end{equation*}
\end{enumerate}
\end{lem}

\begin{proof}
The various norms used below are in the $x$ variable unless otherwise stated. 

Clearly, $\eta_\ep$ and $\psi_\ep$ are  $\Lip $ in time and smooth in space, so $\varphi_\ep$ has the same properties. The function $\varphi_\ep$ is a curl so it is divergence free. Because $\eta_\ep$   vanishes in the neighborhood of   $B(h_\ep(t),\ep)$, so is $\varphi_\ep$. The compact support in space follows immediately once we recall that  $\curl \psi_\ep = \varphi$ and observe that
$$\varphi_{\ep}=\curl(\eta_{\ep} \psi_\ep) =  \eta_{\ep} \varphi + \na \eta_{\ep} \times \psi_\ep$$
Claim {\it(i)} is proved.

To prove {\it(ii)}, we observe that $\text{supp} \; \na \eta_\ep \subset \{ |x-h_\ep(t)|\leq 2\ep\}$ and we estimate 
\begin{align*}
\norm{\varphi_{\ep} - \varphi}_{L^2} 
& \leq  \norm{(\eta_{\ep} -1) \varphi}_{L^2}+ \norm{\na \eta_{\ep} \times \psi_\ep}_{L^2}\\
& \leq  \norm{\eta_{\ep}-1}_{L^2} \norm{\varphi}_{L^{\infty}} + \norm{\na \eta_{\ep}}_{L^2} \norm{\psi_\ep}_{L^{\infty}\left(B(h_\ep (t), 2\ep)\right)}\\
& \leq C \left(  \ep^{\frac{3}{2}} \norm{\varphi}_{L^{\infty}}  +  \ep^{\frac{3}{2}} \norm{\varphi}_{H^{2}}   \right)\\
&\leq  C \ep^{\frac{3}{2}} \norm{\varphi}_{H^2}
\end{align*}
where we used Lemmas \ref{lempsi} and \ref{etaep}.

Taking the supremum on $[0,T]$ we infer that
\begin{equation}\label{estl2}
\| \varphi_{\ep} - \varphi  \|_{L^\infty(0,T;L^{2})} \leq C  \ep^{\frac{3}{2}} \norm{\varphi}_{L^\infty(0,T; H^2 )}.
\end{equation}

Next, by the Sobolev embedding $H^2(\RR^3) \hookrightarrow W^{1,6}(\RR^3)$ and by Lemma \ref{etaep} we estimate
\begin{align*}
\norm{ \na (\varphi_{\ep} - \varphi)}_{L^2} 
& =  \norm{\na \big((\eta_{\ep} -1) \varphi + \na \eta_{\ep} \times \psi_\ep \big)}_{L^2}\\
& \leq  \norm{\eta_{\ep}-1}_{L^3} \norm{\na \varphi}_{L^6} +  \norm{\na \eta_{\ep}}_{L^2} \norm{\varphi}_{L^\infty} + \norm{\na \eta_{\ep}}_{L^2} \norm{\nabla \psi_\ep }_{L^\infty}\\
& \hskip 5cm + \norm{\na^{2}\eta_{\ep}}_{L^2} \norm{\psi_\ep}_{L^{\infty}\left(B(h_\ep (t), 2\ep)\right)} \\
& \leq  C \left(\ep \norm{\varphi}_{W^{1,6}} +  \ep^{\frac{1}{2}} \norm{\varphi}_{L^\infty} + \ep^{\frac{1}{2}} \norm{\nabla \psi_\ep}_{L^\infty} +C\ep^{-\frac12}  \norm{\psi_\ep}_{L^{\infty}\left(B(h_\ep (t), 2\ep)\right)} \right).
\end{align*}

From relations \eqref{def_psiep} and \eqref{var+psi} we get that
\begin{equation*}
 \norm{\nabla \psi_\ep}_{L^\infty} =  \norm{\nabla \psi}_{L^\infty}  \leq C  \norm{\nabla \psi}_{H^2}\leq C\|\varphi\|_{H^2}.
\end{equation*}
From Lemma \ref{lempsi} we have that
\begin{equation*}
\norm{\psi_\ep}_{L^{\infty}\left(B(h_\ep (t), 2\ep)\right)} \leq C \ep  \|\varphi\|_{H^2}. 
\end{equation*}
We conclude from the above relations that
\begin{equation*}
\norm{ \na (\varphi_{\ep} - \varphi)}_{L^2} \leq C\ep^{\frac12} \|\varphi\|_{H^2}. 
\end{equation*}

Taking the supremum on $[0,T]$ we deduce that
\begin{equation}\label{estl22}
\| \na (\varphi_{\ep} - \varphi) \|_{L^\infty(0,T;L^{2})}
 \leq C   \ep^{\frac{1}{2}}   \norm{\varphi}_{L^{\infty}(0,T; H^2)}.
\end{equation}
We conclude form \eqref{estl2} and \eqref{estl22} that
$$\| \varphi_{\ep} - \varphi \|_{L^\infty(0,T;H^{1})} \leq C   \ep^{\frac{1}{2}}   \norm{\varphi}_{L^{\infty}(0,T; H^2)}\to 0 \quad \text{as} \quad \ep\to0.$$

This proves  {\it(ii)}. To prove   {\it(iii)} we simply bound
\begin{equation*}
\|\varphi_{\ep}\|_{L^\infty(0,T;H^{1})} \leq \|\varphi\|_{L^\infty(0,T;H^{1})} +\|\varphi_{\ep}-\varphi \|_{L^\infty(0,T;H^{1})} \leq C  \|\varphi\|_{L^\infty(0,T;H^2)}.
\end{equation*}
This completes the proof of the lemma.
\end{proof}

\section{Temporal estimate and strong convergence }
\label{temp}
The aim of this section is to derive a temporal estimate and to prove the strong convergence of some sub-sequence of $\ve$ in $L^2\loc(\RR_+\times\RR^3)$. We will prove the following result.
\begin{prop}\label{lemstrong}
There exists a sub-sequence $\vepk$ of $\ve$ which converges strongly in $L^2\loc(\RR_+\times\RR^3)$. 
\end{prop}

It suffices to prove that for any $T>0$ there exists a sub-sequence $\vepk$ of $\ve$ which converges strongly in $L^2(0,T;L^2\loc(\RR^3))$. A diagonal extraction then allows to choose the same subsequence for all times $T$. We choose some finite time $T$ and for the rest of this section we assume that $t\in[0,T]$.

The main idea  is to use the Arzelà–Ascoli theorem. Let $\varphi \in C^\infty_{0,\sigma}(\RR^{3})$ be a test function which does not depend on the time. By the definition of the modified stream function, we observe that even if $\varphi$ is constant in time, $\psi_\ep(t,x)$ still depends on the time through $h_\ep(t)$. We construct a family of $\varphi_\ep$ as in Section \ref{defcutof}, so that $\varphi_\ep$ is time-dependent and satisfies Lemma \ref{lemmaphi}.

We first bound
\begin{align*}
\Bigl| \int_{\RR^{3}}\ve(t,x)\cdot\varphi_\ep(t,x)\,dx\Bigr|
&= \Bigl|\int_{\RR^{3}}\ve\cdot\curl (\eta_\ep\psi_\ep)\,dx\Bigr|\\
& =\Bigl|\int_{\RR^{3}}\ve\cdot(\eta_\ep \varphi + \na \eta_\ep\times \psi_\ep)\,dx\Bigr|\\
&\leq \norm{\ve}_{L^2} \norm{\eta_\ep}_{L^\infty} \norm{\varphi}_{L^2} + \norm{\ve}_{L^2} \norm{ \na \eta_\ep}_{L^2} \norm{\psi_\ep}_{L^\infty\left(B(h_\ep(t),2\ep)\right)}\\
& \leq C \left( \norm{\ve}_{L^2} \norm{\varphi}_{L^2} + \ep^{\frac{3}{2}} \norm{\ve}_{L^2}  \norm{\varphi}_{H^2} \right)\\
&\leq C  \norm{\ve}_{L^2} \norm{\varphi}_{H^2} .
\end{align*}
where we used Lemma \ref{lempsi} and Lemma \ref{etaep}. The boundedness of $\ve$ in $L^{2}(\RR^3)$ implies that there exists a constant $C_1$ independent of $\ep $ and $t$ such that   \label{timecont}
\begin{align*}
\Bigl| \int_{\RR^{3}}\ve(t,x)\cdot\varphi_\ep(t,x)\,dx\Bigr|
\leq C_1  \norm{\varphi}_{H^2}.
\end{align*}
We infer that, for $t$ fixed and $\varphi \in C^\infty_{0,\sigma} $, the map
\begin{equation*}
\varphi \mapsto \int_{\RR^{3}}\ve(t,x)\cdot\varphi_\ep(t,x)\,dx\in\RR 
\end{equation*}
is linear and continuous for the $H^2$ norm. Then, there exists some $\Xi_\ep(t)\in H^{-2}_\sigma$ such that
\begin{equation*}
\langle \Xi_\ep(t),\varphi\rangle= \int_{\RR^{3}}\ve(t,x)\cdot\varphi_\ep(t,x)\,dx\qquad\forall\varphi \in H^2_\sigma. 
\end{equation*}
Moreover
\begin{equation}\label{boundpsiep}
\|\Xi_\ep(t)\|_{H^{-2}}\leq C_1\quad \forall t\geq0.  
\end{equation}

From Lemma \ref{lemmaphi} , we know that $\varphi_\ep$ vanishes in the neighborhood of $B(h_\ep(t),\ep)$, so it is compactly supported in the exterior of this ball. Therefore it  can be used as test function in \eqref{nsequation}. Multiplying \eqref{nsequation} by $\varphi_\ep$ and integrating in space and in time from $s$ to $t$ yields
\begin{equation*}
\int_s^t\int_{\RR^{3}}\partial_\tau\ve\cdot\varphi_\ep+ \nu\int_s^t\int_{\RR^{3}}\nabla\ve:\nabla\varphi_\ep+ \int_s^t\int_{\RR^{3}}\ve\cdot\nabla\ve\cdot\varphi_\ep=0.
\end{equation*}
After integrating by parts in time the first term above and using the definition of  $\langle \Xi_\ep(t),\varphi\rangle $, we obtain that
\begin{equation}\label{Xidif}
\langle \Xi_\ep(t)-\Xi_\ep(s),\varphi\rangle=  \int_s^t\int_{\RR^{3}}\ve\cdot\partial_\tau\varphi_\ep
-\nu\int_s^t\int_{\RR^{3}}\na\ve:\na\varphi_\ep
-\int_s^t\int_{\RR^{3}}\ve\cdot\na\ve\cdot\varphi_\ep.
\end{equation}

To bound the second term in the right-hand side above, we recall that $\ve$ is bounded independently of $\ep$ in $L^{\infty}(0,T; L^{2}) \cap L^{2}(0,T; H^{1})$. Thus, by the Hölder inequality and by Lemma \ref{lemmaphi}, we deduce that
\begin{align*}
 \bigl|\nu \int_{s}^{t} \int_{\RR^{3}} \na \ve :\na \varphi_{\ep}\bigr|
&\leq \nu \int_{s}^{t} \norm{\na \ve}_{L^{2}} \norm{\na \varphi_{\ep}}_{L^{2}} \\ 
& \leq C \nu(t-s)^{\frac{1}{2}} \norm{\varphi}_{H^{2}}  \norm{ \ve}_{L^{2}(0,T; H^1)} \\ 
& \leq C\nu (t-s)^{\frac{1}{2}}\norm{\varphi}_{H^2}.
\end{align*}

Next, we estimate the non-linear term in \eqref{Xidif} by the H\"older inequality and by the Gagliardo-Nirenberg inequality $\norm{\ve}_{L^{3}} \leq C \norm{\ve}_{L^{2}}^{\frac{1}{2}} \norm{\na\ve}_{L^{2}}^{\frac{1}{2}}$,
\begin{align*}
|\int_{s}^{t} \int_{\RR^{3}} \ve \cdot \na \ve \cdot \varphi_{\ep}| 
& \leq \int_{s}^{t} \norm{\ve}_{L^{3}}  \norm{\na \ve}_{L^{2}} \norm{\varphi_{\ep}}_{L^6}\\ 
& \leq C \int_{s}^{t} \norm{\ve}_{L^{2}}^{\frac{1}{2}} \norm{\na\ve}_{L^{2}}^{\frac{3}{2}} \norm{\varphi_{\ep}}_{H^1}\\ 
& \leq C (t-s)^{\frac{1}{4}} \norm{\ve}^{\frac{1}{2}}_{L^{\infty}(0,T;L^2)} \norm{\ve}^{\frac{3}{2}}_{L^{2}(0,T;H^1)} \norm{\varphi}_{H^2}\\
& \leq  C (t-s)^{\frac{1}{4}} \norm{\varphi}_{H^2}
\end{align*}

It remains to estimate the term with the time-derivative. Notice that since $\varphi$ does not depend on time, we have that $\partial_{t} \varphi = \partial_{t} \curl \psi_{\ep} = 0$. Several integrations by parts give us
\begin{align*}
\int_s^t\int_{\RR^{3}}\ve\cdot\partial_\tau\varphi_\ep
& = \int_s^t\int_{\RR^{3}}\ve\cdot\curl \partial_\tau  (\eta_\ep \psi_\ep) \\
& = \int_s^t\int_{\RR^{3}}\curl \ve\cdot\partial_\tau (\eta_\ep \psi_\ep) \\
& = \int_s^t\int_{\RR^{3}} \curl \ve\cdot \left(\partial_\tau \eta_\ep \psi_\ep\right) +\int_s^t\int_{\RR^{3}} \curl \ve\cdot \left(\eta_\ep \partial_\tau \psi_\ep \right)\\
& = \int_s^t\int_{\RR^{3}} \curl \ve\cdot \left(\partial_\tau \eta_\ep \psi_\ep\right) +\int_s^t\int_{\RR^{3}} \ve\cdot \curl\left(\eta_\ep \partial_\tau \psi_\ep \right)\\
& = \int_s^t\int_{\RR^{3}} \curl \ve\cdot \left(\partial_\tau \eta_\ep \psi_\ep\right) + \int_s^t\int_{\RR^{3}} \ve\cdot \left(\na \eta_\ep \times \partial_\tau \psi_\ep \right)
\end{align*}
We estimate the two terms in the right-hand side of the equality above by using Lemmas \ref{lempsi} and \ref{etaep} and recalling that $\ep^{\frac{3}{2}} |h'_{\ep}|$ is bounded in $L^\infty(0,T)$ independently of $\ep$:
\begin{align*}
|\int_s^t\int_{\RR^{3}} \curl \ve\cdot \left(\partial_\tau \eta_\ep \psi_\ep\right)| 
& \leq \int_s^t \frac{|h'_\ep|}{\ep} \norm{\curl \ve }_{L^2} \norm{\na \eta \left(\frac{x- h_\ep}{\ep} \right)}_{L^2} \norm{\psi_\ep}_{L^{\infty} (B(h_\ep, 2\ep))}\\
& \leq C \int_s^t \ep^{\frac{3}{2}} |h'_{\ep}| \norm{\curl \ve }_{L^2} \norm{\varphi}_{H^{2} }\\
& \leq C  (t-s)^{\frac{1}{2}}  \norm{ \ve }_{L^2(0,T;H^1)} \norm{\varphi}_{H^{2} } \\
& \leq C (t-s)^{\frac{1}{2}} \norm{\varphi}_{H^{2} }
\end{align*}
and 
\begin{align*}
|\int_s^t\int_{\RR^{3}} \ve\cdot \left(\na \eta_\ep \times \partial_\tau \psi_\ep \right)|
& \leq \int_s^t\norm{\ve}_{L^6} \norm{\na \eta_{\ep}}_{L^{\frac{6}{5}}} \norm{\partial_{\tau} \psi_\ep}_{L^{\infty}(B(h_\ep,2\ep))}\\
& \leq C \int_s^t \ep^{\frac{3}{2}} |h'_\ep| \norm{\ve}_{L^6} \norm{\varphi}_{H^2}\\
& \leq C  (t-s)^{\frac{1}{2}} \norm{ \ve }_{L^2(0,T;H^1)} \norm{\varphi}_{H^{2} } \\
& \leq C (t-s)^{\frac{1}{2}} \norm{\varphi}_{H^{2} }
\end{align*}
where we used the Sobolev embedding $H^1(\RR^3) \hookrightarrow L^{6}(\RR^3)$.

Gathering the two estimates above, we infer that
\begin{align*}
|\int_s^t\int_{\RR^{3}}\ve\cdot\partial_\tau\varphi_\ep | 
& \leq C  (t-s)^{\frac{1}{2}} \norm{\varphi}_{H^2}.
\end{align*}

Putting together all the estimates above yields the following bound for  $\Xi_\ep$: 
\begin{align*}
|\langle \Xi_\ep(t)-\Xi_\ep(s),\varphi\rangle | &\leq  
C\nu  (t-s)^{\frac{1}{2}}\|\varphi\|_{H^2}+C(t-s)^{\frac14}\|\varphi\|_{H^2}+C  (t-s)^{\frac{1}{2}} \norm{\varphi}_{H^2}\\
&\leq C  (t-s)^{\frac{1}{4}} \norm{\varphi}_{H^2}
\end{align*}
where the constant $C$ above depends on $T$ and $\nu$.

By density of $C^{\infty}_{0,\sigma}$ in $H^{2}_\sigma$, we then obtain that $\Xi_\ep(t)$ is equicontinuous in time with value in $H^{-2}_\sigma$
\begin{align*}
\norm{\Xi_\ep(t)-\Xi_\ep(s)}_{H^{-2}} \leq  C (t-s)^{\frac{1}{4}}.
\end{align*}
On the other hand, $\Xi_\ep(t)$ is also bounded in $H^{-2}_\sigma$, see relation \eqref{boundpsiep}. So the compact embedding $H^{-2} \hookrightarrow H^{-3}\loc $ and the Arzelà-Ascoli theorem enable us to extract a subsequence $\Xi_{\ep_k}$ of $\Xi_\ep$ converging to some $\Xi$ strongly in $H^{-3}\loc $:
\begin{equation}\label{strconv}
\Xi_{\ep_k} \to \Xi \quad\text{in } C^0(0,T;H^{-3}\loc).
\end{equation}

We now use Lemmas \ref{lempsi} and  \ref{etaep} to estimate
\begin{align*}
|\langle \Xi_\ep(t)-\ve(t),\varphi\rangle |
&= |\int_{\RR^{3}}\ve(t,x)\cdot\varphi_\ep(t,x) - \int_{\RR^{3}}\ve(t,x)\cdot\varphi(t,x)|\\
&= |\int_{\RR^{3}}\ve(t,x)\cdot (\eta_\ep \varphi + \na \eta_\ep \times \psi_\ep) - \int_{\RR^{3}}\ve(t,x)\cdot\varphi(t,x) | 
\\
& = |\int_{\RR^{3}}(\eta_\ep -1)\ve(t,x) \cdot \varphi + \int_{\RR^{3}} \ve\cdot (\na \eta_\ep \times \psi_\ep) | \\
& \leq  \norm{\ve}_{L^2}\norm{\eta_\ep -1}_{L^2} \norm{\varphi}_{L^\infty} + \norm{\ve}_{L^2} \norm{\na \eta_\ep}_{L^2} \norm{\psi_\ep}_{L^{\infty} (B(h_\ep, 2\ep))}\\
& \leq C \left(\ep^{\frac{3}{2}} \norm{\ve}_{L^2} \norm{\varphi}_{L^\infty} +  \ep^{\frac{3}{2}} \norm{\ve}_{L^2} \norm{\varphi}_{H^2} \right)\\
&\leq C \ep^{\frac{3}{2}}  \norm{\varphi}_{H^2}  \norm{\ve}_{L^2}.
\end{align*}
Using again the density of  $C^{\infty}_{0,\sigma}$ in $H^{2}_\sigma$, the above estimate implies that
\begin{equation*}
\norm{\Xi_\ep(t)-\ve(t)}_{H^{-2}} \leq C \ep^{\frac{3}{2}}  \norm{\ve}_{L^2}. 
\end{equation*}
So $\Xi_\ep-\ve\to0$ in $L^\infty(0,T;H^{-2})$. In particular $\Xi_\ep-\ve\to0$ in $L^\infty(0,T;H^{-3}\loc)$. Recalling \eqref{strconv} and relabelling $\Xi=v$ we infer that 
\begin{equation}\label{convepk}
\vepk \to v \quad \text{in } L^\infty(0,T;H^{-3}\loc).
\end{equation}

Let $f\in C^\infty_0(\RR^3)$.
We have the interpolation inequality
\begin{equation*}
\|f(\vepk-v)\|_{L^2}\leq C\|f(\vepk-v)\|^{\frac14}_{H^{-3}}\|f(\vepk-v)\|^{\frac34}_{H^1}   
\end{equation*}
so
\begin{equation*}
\|f(\vepk-v)\|_{L^{\frac83}(0,T;L^2)}\leq C \|f(\vepk-v)\|^{\frac14}_{L^\infty(0,T;H^{-3})}\|f(\vepk-v)\|^{\frac34}_{L^2(0,T;H^1)}.   
\end{equation*}
Given relation \eqref{convepk} and the boundedness of $\ve$ in $L^2(0,T;H^1)$ we observe that the right-hand side above goes to 0 as $\ep_k\to0$. We deduce that 
\begin{equation*}
\vepk\to v\quad\text{strongly in }  L^{\frac83}(0,T;L^2\loc).
\end{equation*}

The embedding  $L^{\frac83}(0,T;L^2\loc)\subset  L^2(0,T;L^2\loc)$ completes the proof of Proposition \ref{lemstrong}.

\section{Passing to the limit}
\label{paslim}

In this section we are going to complete the proof of Theorem \ref{th2} by passing to the limit with compactness methods. 

Let $T>0$ be finite and fixed. We will pass to the limit only on the time interval $[0,T]$. A diagonal extraction allows us to find a subsequence which converges to the expected limit for all $t\geq0$. 

Thanks to the assumptions on $\ve$, we know that 
\begin{align*}
\ve \; \text{is bounded in} \; L^{\infty}\left(0,T; L^{2}\right) \cap L^{2}\left(0,T; H^{1}\right).
\end{align*}
This implies that there exists some $v \in L^{\infty}\left(0,T; L^{2}\right) \cap L^{2}\left(0,T; H^{1}\right) $ and some sub-sequence $\vepk$ of $\ve$ such that
\begin{gather}
\vepk\rightharpoonup v \quad\text{weak$\ast$ in } L^\infty(0,T;L^2)\notag,\\
\vepk\rightharpoonup v \quad\text{weakly in } L^2(0,T;H^1) \label{weakcon}.
\end{gather}
Moreover, using  Proposition \ref{lemstrong}, we can further assume that
\begin{equation*}
\vepk\to v \quad\text{strongly in } L^2(0,T;L^2\loc). 
\end{equation*}

The main goal of this Section is to prove that the limit $v$ is the solution of the Navier-Stokes equations in $\RR^3$ with initial data $v_0(x)$.

Let $\varphi \in C_{0}^{\infty}([0,T) \times \RR^{3})$ be a divergence-free vector field. 
We construct the family of vector fields $\varphi_{\ep_{k}}$ as in Section \ref{defcutof} (see relation  \eqref{cutof}). These vector fields are compactly supported in the exterior of the ball  $B(h_{\ep_k}(t), \ep_k)$, so they can be used as test functions in \eqref{nsequation}. Multiplying \eqref{nsequation} by $\vepk$ and integrating in space and time yields
\begin{multline} {\label{equalityepk}}
-\int_{0}^{T} \int_{\RR^{3}} \vepk \cdot \pat \varphi_{\ep_{k}} + \nu \int_{0}^{T} \int_{\RR^{3}} \na \vepk : \na \varphi_{\ep_{k}} + \int_{0}^{T} \int_{\RR^{3}} \vepk \cdot \na \vepk \cdot \varphi_{\ep_{k}} \\ = \int_{\RR^{3}} \vepk(0) \cdot \varphi_{\ep_{k}}(0). 
\end{multline}

We will pass to the limit $\ep_k \to 0$ in each of the term in the equation above. First, from Lemma \ref{lemmaphi}, we have that
\begin{align*}
\varphi_{\ep_{k}} (0) \rightarrow \varphi(0) \;\text{strongly in}\; L^{2}(\RR^{3}).
\end{align*}
We also know by hypothesis that $\ve(0,x)$ converges weakly to $v_0(x)$ in $L^2(\RR^3)$.
We infer that
\begin{align}{\label{converge3}}
\int_{\RR^{3}} \vepk (0)  \cdot \varphi_{\ep_{k}}(0)\toepk \int_{\RR^{3}} v(0)  \cdot \varphi(0).
\end{align}

Next, we also know from Lemma \ref{lemmaphi} that 
\begin{align*}
\na \varphi_{\ep_{k}} \rightarrow \na \varphi \;\text{strongly in}\; L^{\infty}(0,T; L^2).
\end{align*}
Recalling that $\nabla \vepk\rightharpoonup\nabla v$ weakly in  $L^{2}([0,T]\times \RR^{3})$, see relation \eqref{weakcon}, 
we deduce that
\begin{align}{\label{converge2}}
\int_{0}^{T}  \int_{\RR^{3}} \na \vepk : \na \varphi_{\ep_{k}} \toepk \int_{0}^{T}  \int_{\RR^{3}} \na v : \na \varphi.
\end{align}

We decompose the non-linear term in the left-hand of \eqref{equalityepk} 
as follows:
\begin{equation*}
\int_{0}^{T} \int_{\RR^{3}} \vepk \cdot \na \vepk \cdot \varphi_{\ep_k}
= \int_{0}^{T} \int_{\RR^{3}} \vepk \cdot \na \vepk \cdot \varphi
+\int_{0}^{T} \int_{\RR^{3}} \vepk \cdot \na \vepk \cdot (\varphi_{\ep_k}-\varphi).
\end{equation*}

To treat the first term on the right-hand side, we know that $\varphi$ is compactly supported, that $\nabla \vepk\rightharpoonup\nabla v$ weakly in  $L^{2}([0,T]\times \RR^{3})$ and that $\vepk\to v $ strongly in $ L^2(0,T;L^2\loc) $. These observations enable us to pass to the limit:
\begin{equation*}
\int_{0}^{T} \int_{\RR^{3}} \vepk \cdot \na \vepk \cdot \varphi
\toepk\int_{0}^{T} \int_{\RR^{3}} v \cdot \na v \cdot \varphi.   
\end{equation*}

For the second term, we make an integration by parts to get that 
\begin{equation*}
	\int_{0}^{T} \int_{\RR^{3}} \vepk \cdot \na \vepk \cdot (\varphi_{\ep_k}-\varphi) = - \int_{0}^{T} \int_{\RR^{3}} \vepk \otimes \vepk  : \na(\varphi_{\ep_k}-\varphi).
\end{equation*}
By the H\"older inequality, the Gagliardo-Nirenberg inequality $\norm{\vepk}_{L^{4}} \leq C \norm{\vepk}_{L^{2}}^{\frac{1}{4}} \norm{\na\vepk}_{L^{2}}^{\frac{3}{4}}$ and the strong convergence of $\varphi_{\ep_k}$ in $L^{\infty}(0,T;H^1)$ stated in Lemma \ref{lemmaphi}, we  obtain that 
\begin{align*}
|- \int_{0}^{T} \int_{\RR^{3}} \vepk \otimes \vepk : \na  (\varphi_{\ep_k}-\varphi)|
& \leq \int_0^T \norm{\vepk}_{L^{4}}^2 \norm{\na (\varphi_{\ep_k}-\varphi)}_{L^2}\\ 
& \leq C \int_0^T \norm{\vepk}_{L^{2}}^{\frac{1}{2}} \norm{\na\vepk}_{L^{2}}^{\frac{3}{2}} \norm{\varphi_{\ep_k}-\varphi}_{H^1}\\ 
& \leq C T^{\frac{1}{4}} \norm{\vepk}^{\frac{1}{2}}_{L^{\infty}(0,T;L^2)} \norm{\vepk}^{\frac{3}{2}}_{L^{2}(0,T;H^1)} \norm{\varphi_{\ep_k}-\varphi}_{L^{\infty}(0,T;H^1)}\\
& \toepk 0
\end{align*}
where we also used the boundedness of $\vepk$ in $L^{\infty}\left(0,T; L^{2}\right) \cap L^{2}\left(0,T; H^{1}\right)$.

 Combining the relations above, we deduce that
 \begin{equation}\label{converge1}
\int_{0}^{T} \int_{\RR^{3}} \vepk \cdot \na \vepk \cdot \varphi_{\ep_k}
\toepk\int_{0}^{T} \int_{\RR^{3}} v \cdot \na v \cdot \varphi.    
 \end{equation}

Now, it remains to pass to the limit in the first term on the left-hand side of \eqref{equalityepk}. Integrating by parts twice allows us to decompose this term into three parts as follows:
\begin{align*}
\int_{0}  ^{T}  \int_{\RR^{3}}\vepk\cdot  \pat\varphi_{\ep_k}
 &= \int_{0}^{T}\int_{\RR^{3}}\vepk\cdot  \curl \pat \left( \eta_{\ep_k} \psi_{\ep_k} \right)\\
&  =  \int_{0}^{T}\int_{\RR^{3}} \curl \vepk \cdot \pat \left(\eta_{\ep_k} \psi_{\ep_k} \right)\\
 & = \int_{0}^{T} \int_{\RR^{3}} \curl \vepk \cdot \left( \pat \eta_{\ep_k} \psi_{\ep_k} \right) + \int_{0}^{T} \int_{\RR^{3}} \curl \vepk \cdot \left( \eta_{\ep_k} \pat \psi_{\ep_k}\right)\\
  & =  \int_{0}^{T} \int_{\RR^{3}} \curl \vepk \cdot \left(\pat \eta_{\ep_k} \psi_{\ep_k} \right) + \int_{0}^{T} \int_{\RR^{3}} \vepk \cdot \curl \left( \eta_{\ep_k} \pat \psi_{\ep_k}\right)\\
   & =  \int_{0}^{T} \int_{\RR^{3}} \curl \vepk \cdot \left( \pat \eta_{\ep_k} \psi_{\ep_k} \right) + \int_{0}^{T} \int_{\RR^{3}}  \vepk \cdot \left(\na \eta_{\ep_k} \times \pat \psi_{\ep_k}\right) \\
&\hskip 6cm +  \int_{0}^{T} \int_{\RR^{3}}  \vepk \cdot \left( \eta_{\ep_k}  \curl \pat \psi_{\ep_k} \right)\\
 & =  \int_{0}^{T} \int_{\RR^{3}} \curl \vepk \cdot \left( \pat \eta_{\ep_k} \psi_{\ep_k} \right) + \int_{0}^{T} \int_{\RR^{3}}  \vepk \cdot \left(\na \eta_{\ep_k} \times \pat \psi_{\ep_k}\right) \\
&\hskip 7cm +  \int_{0}^{T} \int_{\RR^{3}}  \vepk  \cdot \left(  \eta_{\ep_k}  \pat \varphi \right)
\end{align*}
where we used the fact that $ \curl \psi_{\ep_k} = \varphi$ (see Lemma \ref{lempsi}).

We will treat the three terms in the right-hand side of the relation above.  
For the first term,  we use the H{\"o}lder inequality twice, the definition of $\eta_\ep$ (see relation \eqref{defetaep})  and Lemma \ref{lempsi}  to bound
\begin{align*}
|  \int_{0}^{T} \int_{\RR^{3}} \curl \vepk& \cdot \left(  \pat \eta_{\ep_k} \psi_{\ep_k}\right)| \\
& \leq   \int_{0}^{T} \frac{|h'_{\ep_k}(t)|}{\ep_k}    \norm{ \curl \vepk}_{L^{2}} \norm{\na \eta \left( \frac{x-h_{\ep_k}(t)}{{\ep_k}} \right) }_{L^2} \norm{\psi_{\ep_k}}_{L^{\infty} (B(h_{\ep_k}(t), 2\ep_k))} \nonumber \\
&\leq C \int_{0}^{T} \ep_k^{3/2} |h'_{\ep_k}(t)|   \norm{ \curl \vepk}_{L^2}   \norm{\varphi}_{H^{2}} \nonumber \\
&\leq  C T^{\frac{1}{2}}  \ep_k^{3/2} \|h'_{\ep_k}(t)\|_{L^\infty(0,T)} \norm{ \vepk}_{L^2(0,T; H^1)}  \norm{\varphi}_{L^\infty(0,T;H^{2})} \nonumber\\
& \toepk 0.  \nonumber
\end{align*}
where we used the hypothesis $\ep_k^{3/2} h'_{\ep_k}(t)\to 0 $ in $L^\infty\loc (\RR_+)$ when $\ep_k \to 0$.

To bound the second term, we use again the H\"older inequality, Lemmas \ref{lempsi} and \ref{etaep} and the hypothesis on  $\ep_k^{3/2} h'_{\ep_k}$: 
\begin{align*}
|\int_{0}^{T} \int_{\RR^{3}}  \vepk \cdot \left(\na \eta_{\ep_k} \times  \pat \psi_{\ep_k} \right)|
&\leq  \int_{0}^{T}  \norm{\vepk}_{L^{6}} \norm{\na \eta_{\ep_k}}_{L^{\frac{6}{5}}}  \norm{\pat \psi_{\ep_k} }_{L^{\infty}(B(h_{\ep_k}(t), 2\ep_k))} \nonumber\\
& \leq C \int_0^T \ep_k^{3/2} \norm{\vepk}_{L^{6}} \left(  \ep_k\norm{\partial_t\varphi}_{H^2}+|h_{\ep_k}'(t)|\norm{\varphi}_{H^2}\right) \nonumber \\
&\leq C \int_{0}^{T} \norm{\vepk}_{H^{1}}  \left(  \ep_k^{5/2}\norm{\partial_t\varphi}_{H^2}+ \ep_k^{3/2}|h_{\ep_k}'(t)|\norm{\varphi}_{H^2}\right) \nonumber \\
& \leq C T^{\frac{1}{2}}  \norm{\vepk}_{L^2(0,T;H^{1})} \Bigl( 
\ep_k^{5/2} \norm{\partial_t \varphi}_{L^\infty(0,T;H^{2})} \\
&\hskip 3.5cm 
+ \ep_k^{3/2} \|h'_{\ep_k}\|_{L^\infty(0,T)} \norm{\varphi}_{L^\infty(0,T;H^{2})} \Bigr) \nonumber\\
& \toepk 0. \nonumber
\end{align*}
where we also used the Sobolev embedding $H^1(\RR^3) \hookrightarrow L^{6}(\RR^3)$.

For the third term, we shall write $\vepk  \eta_{\ep_k} = \vepk  (\eta_{\ep_k} - 1) + (\vepk -v ) + v$ to get that 
\begin{equation*}
   \int_{0}^{T} \int_{\RR^{3}}  \vepk \cdot \left(\eta_{\ep_k}  \pat \varphi \right) 
   =  \int_{0}^{T} \int_{\RR^{3}}  (\eta_{\ep_k} - 1) \vepk  \cdot \pat \varphi +  \int_{0}^{T}  \int_{\RR^{3}}  (\vepk -v )   \cdot \pat \varphi +  \int_{0}^{T}  \int_{\RR^{3}} v  \cdot \pat \varphi
\end{equation*} 
Recalling  that  $\vepk\to v$ strongly in $ L^2(0,T;L^2\loc)$, we observe that the second term in the right-hand side of the equality above converges to $0$. We estimate the first term by the H\"older inequality and by Lemma \ref{etaep} 
\begin{align*}
  \big{|} \int_{0}^{T} \int_{\RR^{3}} (\eta_{\ep_k} - 1)\vepk   \cdot \pat \varphi \big{|}
  & \leq \int_{0}^{T} \norm{\vepk}_{L^2} \norm{\eta_{\ep_k} - 1}_{L^2} \norm{ \pat \varphi}_{L^\infty} \\
  & \leq C \ep_k^{\frac{3}{2}} \int_{0}^{T} \norm{\vepk}_{L^2} \norm{ \pat \varphi}_{L^\infty} \\
  & \leq C T\ep_k^{\frac{3}{2}} \norm{\vepk}_{L^\infty(0,T;L^2)} \norm{ \partial_t\varphi}_{L^\infty(0,T; L^\infty)} \\
&\toepk 0 .
\end{align*}
We infer that  
\begin{equation*}
\int_{0}^{T} \int_{\RR^{3}}  \vepk \cdot \left(\eta_{\ep_k}  \pat \varphi \right) \toepk \int_{0}^{T} \int_{\RR^{3}} v \cdot \pat \varphi,
\end{equation*}
which implies that 
\begin{equation}{\label{converge5}}
\int_{0}  ^{T}  \int_{\RR^{3}}\vepk\cdot  \pat\varphi_{\ep_k}
 \toepk \int_{0}^{T} \int_{\RR^{3}} v \cdot \pat \varphi.
\end{equation}

Gathering \eqref{equalityepk}, \eqref{converge3}, \eqref{converge2}, \eqref{converge1} and \eqref{converge5}, we conclude that 
\begin{equation*}
-\int_{0}^{T} \int_{\RR^{3}} v \cdot \partial_{t}\varphi 
 + \nu \int_{0}^{T} \int_{\RR^{3}} \nabla v :\nabla \varphi + \int_{0}^{T} \int_{\RR^{3}} v\cdot \nabla v\cdot \varphi =  \int_{\RR^{3}} v(0) \cdot\varphi(0)
\end{equation*}
which is the weak formulation of the Navier-Stokes equations in $\RR^3$. This completes the proof that $v$ is a solution of the Navier-Stokes equations in $\RR^3$ in the sense of  distributions.

\medskip

In order to complete the proof of Theorem \ref{th2}, it remains to prove the energy inequality \eqref{ineq_vep} under the additional assumption that $\ve(0)$ converges strongly to $v_0$ in $L^2$.

Let us observe first that $v\in{C}_{w}^{0}(\RR_+; L^{2}(\RR^{3}))$. This follows from the fact that $v\in L^\infty\loc(\RR_+;L^2(\RR^3))$ is a solution in the sense of distributions of the Navier-Stokes equations in $\RR^3$. The argument is classical, but let us recall it for the benefit of the reader. We apply the Leray projector $\PP$ in $\RR^3$ to the Navier-Stokes equations verified by $v$ to obtain that
\begin{equation*}
\partial_t v-\nu\Delta v+\PP\dive(v\otimes v)=0.  
\end{equation*}
Because $v\in L^\infty\loc(\RR_+;L^2(\RR^3))$ we have that $v\otimes v\in L^\infty\loc(\RR_+;L^1(\RR^3))\subset L^\infty\loc(\RR_+;H^{-2}(\RR^3))$. We also have that $\Delta v\in  L^\infty\loc(\RR_+;H^{-2}(\RR^3))$, so $\partial_t v\in L^\infty\loc(\RR_+;H^{-2}(\RR^3))$. We infer that $v$ is Lipschitz in time with values in $H^{-2}(\RR^3)$; in particular it is strongly continuous in time with values in $H^{-2}(\RR^3)$. This strong continuity in time together with the boundedness of the $L^2$ norm implies the weak continuity of $v$ in time with values in $L^2$. In particular, we have that $v(t)$ is well-defined and belongs to $L^2(\RR^3)$ for all times $t\geq0$ (and not only for almost all times).

Let us observe now that for all $t\geq0$ we have that $\vepk(t)\rightharpoonup v(t)$ weakly in $L^2(\RR^3)$. Indeed, we know from relation \eqref{convepk} that $\vepk(t)\to v(t)$ strongly in $H^{-3}\loc(\RR^3)$, so $\langle \vepk(t),\varphi\rangle\to \langle v(t),\varphi\rangle$ for all test functions $\varphi$. The boundedness of $\vepk(t)$ in $L^2$ and the density of the test functions in $L^2$ imply that $\langle \vepk(t),\varphi\rangle\to \langle v(t),\varphi\rangle$ for all  $\varphi\in L^2$, that is $\vepk(t)\rightharpoonup v(t)$ weakly in $L^2(\RR^3)$.

Let us denote by $\chi_A$ the characteristic function of  the set $A$. We prove now that for all $t\geq0$
\begin{equation*}
\vepk(t)\chi_{B(h_\epk(t),\epk)}  \rightharpoonup 0
\end{equation*}
weakly in $L^2$. Indeed, let $g\in L^2$. Then
\begin{equation*}
|  \langle \vepk(t)\chi_{B(h_\epk(t),\epk)}, g\rangle|=|\int_{B(h_\epk(t),\epk)} \vepk(t)g|\leq \|\vepk(t)\|_{L^2}\|g\|_{L^2(B(h_\epk(t),\epk))}\toepk0
\end{equation*}
because $\|\vepk(t)\|_{L^2}$ is bounded and $\|g\|_{L^2(B(h_\epk(t),\epk))}$ goes to 0 as $\epk\to0$.

We infer that
\begin{equation*}
\vepk(t)\chi_{\RR^3\setminus B(h_\epk(t),\epk)}  \rightharpoonup  v(t) \quad\text{weakly in }L^2(\RR^3).
\end{equation*}
By the weak  lower semi-continuity of the $L^2$ norm we infer that
\begin{equation}\label{weakcv1}
\|v(t)\|_{L^2(\RR^3)}\leq \liminf_{\epk\to0}\|\vepk(t)\|_{L^2(\RR^3\setminus B(h_\epk(t),\epk))}.  
\end{equation}

Similarly, from the weak convergence 
\begin{equation*}
D( \vepk) \rightharpoonup  D(v) \quad\text{weakly in }L^2((0,t)\times\RR^3)  
\end{equation*}
we infer that
\begin{equation*}
\chi_{\RR^3\setminus B(h_\epk(t),\epk)} D(\vepk) \rightharpoonup D(v) \quad\text{weakly in }L^2((0,t)\times\RR^3)  
\end{equation*}
so by lower semi-continuity 
\begin{equation}\label{weakcv2}
\|D(v)\|^2_{L^2((0,t)\times\RR^3)}\leq \liminf_{\epk\to0}\int_0^t\|D(\vepk)\|^2_{L^2(\RR^3\setminus B(h_\epk(t),\epk))}.  
\end{equation}

We also observe at this point that the strong $L^2$ convergence of  $\ve(0)$ towards $v_0$ gives that 
\begin{equation}\label{weakcv3}
\|v_0\|_{L^2}=\liminf_{\epk\to0}\|\vepk(0)\|_{L^2}.  
\end{equation}

Finally, taking the $\liminf\limits_{\epk\to0}$ in \eqref{ineqvep} and using \eqref{weakcv1}, \eqref{weakcv2} and \eqref{weakcv3} implies the required energy inequality \eqref{ineq_vep}. This completes the proof of Theorem \ref{th2}.

\begin{remark}
We end this paper with a final remark about the weak time continuity assumed in Theorem \ref{th2}: $v_\ep\in C^0_w(\RR_+; L^{2}(\RR^{3}))$. We know that  $\wep\in C^0_w(\RR_+; L^{2}(\RR^{3}))$ so we made this hypothesis for the sake of simplicity, but it is in fact not necessary to make such an assumption. Indeed, we used it to make sense of the various terms of the form $\int _{\RR^3}\ve(t,x)\cdot \varphi_\ep(t,x)\,dx$, see for instance on page \pageref{timecont}. But the Navier-Stokes equation itself implies a time-continuity property allowing to make sense of such terms.  More precisely, let us make a change of variables to go to a fixed domain. The vector field $\vet(t,x)=\ve(t,x-h_\ep(t))$ verifies the following PDE: 
 \begin{equation}\label{nsequationtilde}
\partial_t\vet+h'_\ep\cdot\nabla\vet-\nu\Delta\vet+\vet\cdot\nabla\vet=-\nabla\widetilde\pi_\ep \quad\text{for }|x|>\ep
  \end{equation}
where $\widetilde\pi_\ep(t,x)=\pi_\ep(t,x-h_\ep(t))$. Since $\ve\in L\loc ^{\infty}(\RR_+; L^{2}(\RR^{3})) \cap L^{2}\loc(\RR_+; H^{1}(\RR^{3}))$ we also have that $\vet\in L\loc ^{\infty}(\RR_+; L^{2}(\RR^{3})) \cap L^{2}\loc(\RR_+; H^{1}(\RR^{3}))$. Recalling that $h_\ep$ is Lipschitz, we infer by classical estimates that $h'_\ep\cdot\nabla\vet-\nu\Delta\vet+\vet\cdot\nabla\vet\in L\loc ^{\frac43}(\RR_+; H^{-1}(\RR^{3}))$. So, if we choose some $\Phi \in C^\infty_{0,\sigma}(|x|>\ep)$ and we multiply \eqref{nsequationtilde} by $\Phi$, the pressure goes away and we obtain that
\begin{equation*}
|\langle \partial_t\vet, \Phi\rangle|=|\langle h'_\ep\cdot\nabla\vet-\nu\Delta\vet+\vet\cdot\nabla\vet,\Phi\rangle|\leq \|h'_\ep\cdot\nabla\vet-\nu\Delta\vet+\vet\cdot\nabla\vet\|_{H^{-1}}\|\Phi\|_{H^1}.
\end{equation*}
If we denote by $X$ the dual of $C^\infty_{0,\sigma}(|x|>\ep)$ for the $H^1$ norm, the relation above implies that
\begin{equation*}
\|\partial_t\vet\|_X\leq \|h'_\ep\cdot\nabla\vet-\nu\Delta\vet+\vet\cdot\nabla\vet\|_{H^{-1}}
\end{equation*}
so $\partial_t\vet\in L\loc ^{\frac43}(\RR_+;X)$. In particular $\vet\in C^0(\RR_+;X)$. We infer that $\int \vet(t,\cdot)\cdot\Phi$ is well-defined for all $t\geq0$ and  $\Phi \in C^\infty_{0,\sigma}(|x|>\ep)$. Going back to the original variables, we infer that if $\varphi_\ep \in C^\infty_{0,\sigma}(|x-h_\ep(t)|>\ep)$ then  $\int \ve(t,\cdot)\cdot\varphi_\ep$ is well-defined. 

\medskip

\end{remark}

\section*{Acknowledgments}   J.H. and D.I. have been partially funded by the ANR project Dyficolti ANR-13-BS01-0003-01. D.I. has been  partially funded by the LABEX MILYON (ANR-10-LABX-0070) of Université de Lyon, within the program ``Investissements d'Avenir'' (ANR-11-IDEX-0007) operated by the French National Research Agency (ANR).

\bigskip

\begin{description}
\item[J. He] Université de Lyon, Université Lyon 1 --
CNRS UMR 5208 Institut Camille Jordan --
43 bd. du 11 Novembre 1918 --
Villeurbanne Cedex F-69622, France.\\
Email: \texttt{jiao.he@math.univ-lyon1.fr}
\item[D. Iftimie] Université de Lyon, Université Lyon 1 --
CNRS UMR 5208 Institut Camille Jordan --
43 bd. du 11 Novembre 1918 --
Villeurbanne Cedex F-69622, France.\\
Email: \texttt{iftimie@math.univ-lyon1.fr}
\end{description}

\end{document}